\documentclass[reqno,10pt]{amsart}
\usepackage{amsmath,amsthm,enumerate,graphics,amsfonts,latexsym,amsopn,verbatim,amscd,amssymb}
 
\theoremstyle{plain}
\newtheorem{thm}{Theorem}[section]

\newtheorem{cor}[thm]{Corollary}
\newtheorem{prop}[thm]{Proposition}

\newtheorem{problem}[thm]{Problem}
\newenvironment{prob}{\problem\rm}{\endproblem}

\theoremstyle{definition}
\newtheorem{rem}[thm]{Remark}

\DeclareMathOperator{\Sol}{Sol}
\DeclareMathOperator{\diam}{diam}

\DeclareMathOperator{\girth}{girth}
\DeclareMathOperator{\PSL}{PSL}
\DeclareMathOperator{\Sz}{Sz}
\DeclareMathOperator{\GF}{GF}
\newcommand{\PGamL}{\mathop{\mathrm{P}\Gamma\mathrm{L}}}

\begin{document}
\title[Solvable conjugacy class graph of groups]{Solvable conjugacy class graph of groups}

\author[Bhowal, Cameron, Nath and Sambale]{Parthajit Bhowal, Peter J. Cameron, Rajat Kanti Nath and Benjamin Sambale}

\address{P. Bhowal, Department of Mathematics, Tezpur University, Napaam-784028, Sonitpur, Assam, India. Department of Mathematics, Cachar College, Silchar-788001, Assam, India.}
\email{bhowal.parthajit8@gmail.com}

\address{P. J. Cameron, School of Mathematics and Statistics, University of St Andrews, North Haugh, St Andrews, Fife, KY16 9SS, UK.}
\email{pjc20@st-andrews.ac.uk}

\address{R. K. Nath, Department of Mathematics, Tezpur University, Napaam-784028, Sonitpur, Assam, India.}
\email{rajatkantinath@yahoo.com}

\address{B. Sambale, Institut f\"ur Algebra, Zahlentheorie und Diskrete Mathematik, Leibniz Universit\"at Hannover, 30167 Hannover, Germany}
\email{sambalemath.uni-hannover.de}

\begin{abstract}
In this paper we introduce the graph $\Gamma_{sc}(G)$ associated with  a group $G$, called the solvable conjugacy class graph (abbreviated as SCC-graph), whose  vertices are the nontrivial conjugacy classes of $G$ and two distinct  conjugacy classes $C, D$ are adjacent if there exist $x \in C$ and $y \in D$ such that $\langle x, y\rangle$ is solvable.

We discuss the connectivity, girth, clique number, and several other properties of the SCC-graph. One of our results asserts that there are only finitely many finite groups whose SCC-graph has given clique number~$d$, and we find explicitly the list of such groups with $d=2$. We pose some problems on the relation of the SCC-graph to the solvabale graph and to the NCC-graph, which we cannot solve.
\end{abstract}

\subjclass[2010]{05C25, 20E45, 20F16}
\keywords{Graph, conjugacy class, non-solvable group, clique number}

\maketitle

\section{Introduction} \label{S:intro}
Various graphs have been defined on finite groups, using various properties of groups, and studied to understand the interplay between groups and graphs.    Commutativity is one  of such graph defining properties on groups. The first graph arising from commutativity is the commuting graph \cite{BF1955} which is the complement of the non-commuting graph \cite{EN76}.
The commuting graph of a group $G$ is a graph whose vertices are the nontrivial elements of $G$ and two distinct vertices $x, y$ are adjacent if $\langle x, y \rangle$ is abelian. 

Another way of defining graphs on $G$ is given by considering the conjugacy classes as the vertex set and adjacency is defined using properties of conjugacy classes \cite{BHM90}. We write $x^G$ to denote the conjugacy class $\{gxg^{-1} : g \in G\}$ of $x \in G$.  Mixing the concepts of commutativity and conjugacy class, in 2009 Herzog et al. \cite{hlm} have introduced  commuting conjugacy class graph (abbreviated as CCC-graph) of  $G$ as a graph whose vertex set is the set of nontrivial conjugacy classes of $G$ and two distinct vertices $x^G$ and $y^G$ are adjacent if $\langle x^{\prime}, y^{\prime}\rangle$ is abelian for some $x^{\prime}\in x^G$ and $y^{\prime}\in y^G$.

Extending the notion of CCC-graph, in 2017 Mohammadian and Erfanian \cite{ME2017} introduced the nilpotent conjugacy class graph (abbreviated as NCC-graph) of a group. Its vertex set is the set of nontrivial conjugacy classes of $G$, and two distinct vertices $x^G$ and $y^G$ are adjacent if $\langle x^{\prime}, y^{\prime}\rangle$ is nilpotent for some $x^{\prime}\in x^G$ and $y^{\prime}\in y^G$. 
Note that the CCC-graph  is a spanning subgraph of the NCC-graph of $G$.

 In this paper, we further extend the notions of CCC-graph and NCC-graph and introduce the solvable conjugacy class graph (abbreviated as SCC-graph) of $G$. The SCC-graph of a group $G$ is a simple undirected graph,  denoted by $\Gamma_{sc}(G)$, with  vertex set  $\{x^G : 1 \ne x \in G\}$ and two distinct vertices $x^G$ and $y^G$ are adjacent if there exist two elements $x^{\prime}\in x^G$ and $y^{\prime}\in y^G$ such that $\langle x^{\prime}, y^{\prime}\rangle$ is solvable. It is clear that the NCC-graph  is a spanning subgraph of the SCC-graph of $G$.

The aim of  this paper is to discuss the connectivity, girth, clique and some properties of SCC-graph. 

For each of these conjugacy class graphs, we sometimes consider a variant where
the vertex set is $G$, rather than the set of conjugacy classes in $G$, with
the adjacency rules as described. So, if two conjugacy classes are equal or
adjacent in the conjugacy class version, then all pairs of vertices in those
classes are adjacent in the graph on $G$. This change leaves several 
properties, such as connectedness and diameter, unchanged. The reason for
doing this is that we want to compare the solvable conjugacy class graph with
the usual solvable graph (see \cite{bn-nath-2020}) on $G$, which
is easier if the vertex sets are the same. We will call the version of the
SCC-graph with vertex set $G$ the \emph{expanded SCC-graph} on $G$. The expanded CCC-graph and NCC-graph are defined analogously.

We write $V(\Gamma)$ to denote the vertex set of a graph $\Gamma$.  We write $u \sim v$ to denote that the vertices $u, v$ are adjacent. 

The distance between two vertices $u$ and $v$ of $\Gamma$ is denoted by $d(u, v)$.  Recall that the diameter of a graph is the maximum distance between its vertices. If $u = v$ then we write $d(u, v) = 0$.
The solvabilizer of $x$, denoted by $\Sol_G(x)$, is the set given by $\{y \in G :\langle x, y\rangle \text{ is solvable}\}$.

We write $\Sol(G) = \{x \in G :\langle x, y\rangle \text{ is solvable for all } y \in G\}$. Clearly,  $\Sol(G) = \bigcap_{ u\in G}\Sol_G(u)$ and $Z(G) \subseteq \Sol(G)$. Also, if $G$ is finite then $\Sol(G)$ is the solvable radical of $G$ (see \cite{gkps}).

We note that the second and third authors, together with Arunkumar and
Selvaganesh, have considered  graphs defined by
combining a graph on the group (such as the commuting graph) with an
equivalence class (such as conjugacy) in \cite{acns}. However, the SCC-graph is not
considered in that paper.

We begin with a simple observation. Let $a$ and $b$ be two elements of $G$
such that $a^G$ and $b^G$ are joined in the SCC-graph of $G$. This means that
there exist $a'\in a^G$ and $b'\in b^G$ such that $\langle a',b'\rangle$ is
solvable. Without loss of generality, we can assume that $a'=a$. For suppose
that $(a')^h=a$. Then $\langle a',b'\rangle^h=\langle a,(b')^h\rangle$ is
solvable, since it is a conjugate of (and hence isomorphic to)$\langle a',b'\rangle$.

\section{Properties of the SCC-Graph}

\begin{thm}
Let $G$ be a finite group. Then the SCC-graph of $G$ is complete if and only
if $G$ is solvable.
\label{t21}
\end{thm}

\begin{proof}
If $G$ is solvable, $\langle x, y\rangle$ is also solvable for all $x, y\in G$. In particular, if  $a^G, b^G$ are two vertices of $\Gamma_{sc}(G)$ and $x \in a^G$, $y \in b^G$ then $\langle x, y\rangle$ is solvable.  Therefore, $a^G$ and $b^G$ are adjacent. Hence, $\Gamma_{sc}(G)$ is a complete graph. 

Conversely, suppose that $\Gamma_{sc}(G)$ is complete. Then, by the observation
at the end of the last section, for every $a,b\in G$, there is a conjugate
$b'$ of $b$ such that $\langle a,b'\rangle$ is solvable. By 
\cite[Theorem A]{dghp}, we conclude that $G$ is solvable.
\end{proof}

\begin{thm}\label{solv}
Let $G$ be a finite solvable group with complete expanded NCC-graph. Then $G$ is nilpotent.
\end{thm}

\begin{proof}
Suppose $G$ is not nilpotent. Let $N$ be a minimal normal subgroup of $G$. Then $N$ is an elementary abelian $p$-group. If $N$ is not a Sylow $p$-subgroup, then $G/N$ is not nilpotent since $O_p(G/N)=1$. Clearly, the expanded NCC-graph of $G/N$ is still complete. Hence, we can replace $G$ by $G/N$. Continuing in this way, we end up with a minimal normal subgroup $N$ which is a Sylow $p$-subgroup and $G/N$ is nilpotent. We can further assume that $O_{p'}(G)=1$, so that $G/N$ acts faithfully on $N$. Let $x\in G$ be a $p'$-element such that $1\ne xN\in Z(G/N)$. Then there exists $y\in N$ such that $xy\ne yx$. It follows that $\langle x,y\rangle$ is not nilpotent. Now for every $g\in G$ we have $gxg^{-1}\equiv x\pmod{N}$. Hence, also $\langle gxg^{-1},y\rangle$ is not nilpotent. Therefore, $x$ and $y$ are not adjacent in the expanded NCC-graph. Contradiction. 
\end{proof}

For sake of discussion suppose that $G$ is a non-solvable group such that the expanded SCC-graph and NCC-graph coincide. Let $S\unlhd G$ be the solvable radical. If $xS,yS\in G/S$ are adjacent in the expanded SCC-graph of $G/S$, then there exists $g\in G$ such that $H:=\langle x,gyg^{-1}\rangle$ is solvable. By Theorem~\ref{solv}, there exists $h\in H$ such that $\langle x,hgyg^{-1}h^{-1}\rangle$ is nilpotent. Then also $\langle xS,hgy(hg)^{-1}S\rangle$ is nilpotent and $xS$ is adjacent to $yS$ in the expanded NCC-graph of $G/N$. Hence, we may assume that $S=1$. Now let $N$ be a minimal normal subgroup of $G$. Then $N=T_1\times\cdots\times T_n$ for isomorphic non-abelian simple groups $T_1,\ldots,T_n$. Let $T:=\{(t,\ldots,t)\}\le N$ be a diagonal subgroup. Then elements in $T$ are conjugate in $G$ if and only if they are conjugate in $\mathrm{Aut}(T)$. In order to derive a contradiction we may replace $G$ by $\mathrm{Aut}(T)$, i.e. we assume that $G$ is an almost simple group. 

Let $A_n\le G\le S_n$ and let $p$ be the largest prime $\le n$. By Bertrand's Postulate, $n<2p$. Let $x\in A_n$ be a $p$-cycle and $y\in N_{A_n}(\langle x\rangle)$ a disjoint product of a $(n-1)$-cycle and a transposition such that $y$ generates $\mathrm{Aut}(\langle x\rangle)$. Clearly $x$ and $y$ are adjacent in the expanded SCC-graph. Suppose that there exists $g\in G$ such that $\langle x,gyg^{-1}\rangle$ is nilpotent. Since $x$ and $y$ have coprime orders, it follows that $x$ commutes with $gyg^{-1}$. On the other hand, $y$ and $gyg^{-1}$ have the same cycle type. Hence, the cycles of $gyg^{-1}$ must be disjoint to the $p$-cyclic $x$. This is impossible since $p+(p-1)+2>2p> n$. 

At this point we record several questions about SCC-graphs, which we have not
been able to answer, together with some comments.

\begin{prob}
Given a finite group $G$, describe the set of vertices of the expanded SCC-graph
of $G$ which are joined to all others.
\end{prob}

We observe that, in the solvable graph, the set of dominant vertices is just
the solvable radical of $G$, by the result of \cite{gkps}. Hence the set of
dominant vertices in the expanded SCC-graph contains the solvable radical.
However, it can be larger. Consider the simple groups $\PSL(2,2^d)$,
with $d\ge2$. Each group has a unique conjugacy class of involutions, and every
element of the group is mapped to its inverse by conjugation by some involution.
So, for any element $a\in G$, there is an involution $b\in G$ so that
$\langle a,b\rangle$ is dihedral, and hence solvable. Thus the involutions are
dominant vertices. On the other hand, $G$ is non-abelian simple, so its
solvable radical is trivial. The sporadic Janko group $J_1$ also has this
property.

\begin{prob}
\begin{enumerate}
\item
For which finite groups $G$ is the expanded SCC-graph of $G$ equal to the
solvable graph of $G$?
\item 
For which finite non-solvable groups $G$ is the expanded SCC-graph of $G$ equal to the
expanded NCC-graph of $G$?
\end{enumerate}
\end{prob}

It is known that
\begin{itemize}
\item the expanded CCC-graph is equal to the commuting graph if and only if $G$
is a $2$-Engel group (that is, satisfies the commutator identity
$[x,y,y]=1$ for all $x,y\in G$ \cite[Theorem 2.2]{acns};
\item the solvable graph is equal to the nilpotent graph if and only if
$G$ is nilpotent~\cite[Proposition 11.1(b)]{gong}.
\end{itemize}
It may be that the answers to these two questions are ``$G$ is solvable'' and
``$G$ is nilpotent'' respectively.

We give here another open problem, which is loosely related to the above
problems.

\begin{prob}
For which finite graphs $\Gamma$ is there a finite group $G$ such that
$\Gamma$ is isomorphic to an induced subgraph of $\Gamma_{sc}(G)$?
\end{prob}

We note that every finite graph can be embedded in the solvable graph of some
finite group (see~\cite[p.~93]{gong}), but this construction does not descend
to the solvable conjugacy class graph.

\medskip

Next we turn to the questions of connectedness and diameter. The girth will
be discussed in the next section, but we begin with a simple observation.

\begin{prop} \label{aa}
Let $G$ be a  non-solvable group such that it has an element of order $pq$, where $p, q$ are primes.
If $p\neq q$    then $\girth(\Gamma_{sc}(G)) = 3$ and hence  $\Gamma_{sc}(G)$ is not a tree.
\end{prop}

\begin{proof}
Let $a \in G$ be an element of  order $pq$. 
If $p\neq q$  then  $o(a^q) = p$ and  $o(a^p) = q$. Also, $\langle a, a^q\rangle, \langle a^q, a^p\rangle$ and $\langle a^p, a\rangle$ are abelian groups. Since $a^G$, $(a^q)^G$ and $(a^p)^G$ are distinct, we have the following triangle
\[
a \sim a^q \sim a^p \sim a
\]
in $\Gamma_{sc}(G)$. Therefore, $\girth(\Gamma_{sc}(G)) = 3$ and hence  $\Gamma_{sc}(G)$ is not a tree.
\end{proof}

\begin{prop}
Let  $x\in G\setminus \{1\}$ and $a, b\in \Sol_G(x)\setminus \{1\}$. Then $a^G$ and $b^G$ are connected and $d(a^G, b^G)\leq 2$.
In particular, if $\Sol(G)\neq \{1\}$ then $\Gamma_{sc}(G)$ is connected and $\diam(\Gamma_{sc}(G))\leq 2$.
\end{prop}
\begin{proof}
Since $a, b\in \Sol_G(x)\setminus \{1\}$, $\langle a, x\rangle$ and $\langle x, b\rangle$ are solvable. Therefore,  $d(a^G, x^G)\leq 1$ and $d(x^G, b^G)\leq 1$. Hence,  the result follows.

If $\Sol(G)\neq \{1\}$ then there exists an element $z \in G$ such that   $z \neq 1$ and $z \in \Sol(G)$. Therefore, $z\in \Sol_G(w)$ for all $w \in G \setminus \{1\}$. Let  $u^G$ and $v^G$ be any two vertices of $\Gamma_{sc}(G)$. Then $u, v \in \Sol_G(z)\setminus \{1\}$. Therefore, by the first part it follows that  $d(u^G, v^G)\leq 2$. Hence, $\diam(\Gamma_{sc}(G))\leq 2$.
\end{proof}

\begin{rem}
For any two distinct vertices $x^G, y^G\in V(\Gamma_{sc}(G))$, $x^G \sim y^G$ if and only if $\Sol_G(gxg^{-1}) \cap y^G \neq \emptyset$ for all $g \in G$. Also, $x^G$ is an isolated vertex if and only if $\Sol_G(gxg^{-1})\subseteq x^G \cup \{1\}$ for all $g\in G$.
\end{rem}

\begin{thm}
If $G, H$ are arbitrary nontrivial groups then the graph $\Gamma_{sc}(G \times H)$ is connected and $\diam(\Gamma_{sc}(G \times H))\leq 3$. In particular, $\Gamma_{sc}(G \times G)$ is connected and $\diam(\Gamma_{sc}(G \times G))\leq 3$. Further, $\diam(\Gamma_{sc}(G \times G)) = 3$  if and only if either $\Gamma_{sc}(G)$ is disconnected or $\Gamma_{sc}(G)$ is connected with  $\diam(\Gamma_{sc}(G)) \geq 3$.
\end{thm}
\begin{proof}
Let $(x, y)$ and $(u, v)$ be two nontrivial elements of $G \times H$. Without any loss we may assume that  $x \neq 1_G$ and $v \neq 1_H$, where $1_G$ and $1_H$ are identity elements of $G$ and $H$ respectively, then
\begin{equation*}
(x,y)^{G\times H} \sim (x,1_H)^{G\times H} \sim (1_G,v)^{G\times H} \sim (u,v)^{G\times H}.
\end{equation*}
This shows that $\Gamma_{sc}(G \times H)$ is connected and $\diam(\Gamma_{sc}(G \times H))\leq 3$. Putting $H = G$, it follows that $\Gamma_{sc}(G \times G)$ is connected and $\diam(\Gamma_{sc}(G \times G))\leq 3$.

Let $\diam(\Gamma_{sc}(G \times G)) = 3$.  Suppose that $\Gamma_{sc}(G)$ is connected and  $\diam(\Gamma_{sc}(G))\leq 2$ (on the contrary).
Let $(x, y), (u, v)$ be two vertices  in $\Gamma_{sc}(G \times G)$. Without any loss we may assume that $x, u\neq 1_G$. Since $\Gamma_{sc}(G)$ is connected and  $\diam(\Gamma_{sc}(G))\leq 2$,  there exist $a \in G \setminus \{1_G\}$ such that $x^G \sim a^G \sim u^G$. Therefore,  $\langle x^f, a^g\rangle$ and $\langle a^h, u^w\rangle$ are solvable for some  $f, g, h, w \in G$. We have $\langle (x, y)^{(f, c)}, (a, 1_G)^{(g, d)}\rangle = \langle x^f, a^g\rangle \times \langle y^c \rangle$, where $c, d \in G$. Since $\langle x^f, a^g\rangle$ and $\langle y^c \rangle$ are solvable, $(x, y)^{G\times G}\sim (a, 1_G)^{G\times G}$. Similarly,  $(u, v)^{G\times G} \sim (a, 1_G)^{G\times G}$. Thus we get the following path  
\begin{equation*}
 (x, y)^{G\times G}\sim (a, 1_G)^{G\times G}\sim (u, v)^{G\times G}.
\end{equation*}
Therefore, $\diam(\Gamma_s(G \times G)) \leq 2$, which is a contradiction. Hence,  $\Gamma_{sc}(G)$ is disconnected or $\Gamma_{sc}(G)$ is connected with $\diam(\Gamma_{sc}(G)) \geq 3$.

Suppose that either $\Gamma_{sc}(G)$ is disconnected or it is connected with $\diam(\Gamma_s(G)) \geq 3$. 
Then there exist two distinct elements $x, y \in G \setminus \{1_G\}$ such that either $x^G, y^G$ are not connected or $d(x^G, y^G) \geq 3$. We are to show that $\diam(\Gamma_{sc}(G \times G)) = 3$. Suppose that $\diam(\Gamma_{sc}(G \times G)) \leq 2$. Consider the following two cases. 

\noindent \textbf{Case 1.}  $\Gamma_{sc}(G)$ is disconnected.

Let $u^G$ and $v^G$ be any two distinct vertices in $\Gamma_{sc}(G)$. Then   $d((u, 1_G)^{G\times G}, (v, 1_G)^{G\times G}) = 1$ or $2$. If $d((u, 1_G)^{G\times G}, (v, 1_G)^{G\times G}) = 1$ then $(u, 1_G)^{G\times G} \sim (v, 1_G)^{G\times G}$. Therefore, $\langle u^f, v^w\rangle$ is solvable for some  $f, w \in G$. Therefore,   $u^G\sim  v^G$ and so $d(u^G, v^G) = 1$; a contradiction. 
If $d((u, 1_G)^{G\times G}, (v, 1_G)^{G\times G}) = 2$ then there exists a non-identity element   $(a, b)\in G \times G$  such that
\begin{equation*}
(u, 1_G)^{G\times G}\sim (a,b)^{G\times G}\sim (v, 1_G)^{G\times G}.
\end{equation*}
It follows that $\langle u^f, a^g\rangle$ and $\langle a^h, v^w\rangle$ are solvable for some  $f, g, h, w \in G$ and so
 \begin{equation*}
 u^G\sim a^G \sim v^G.
 \end{equation*} 
Thus  $u^G$, $v^G$ are connected and $d(u^G, v^G) \leq 2$, a contradiction. 

\noindent \textbf{Case 2.}  $\Gamma_{sc}(G)$ is connected with $\diam(\Gamma_s(G)) \geq 3$.

Proceeding as in Case 1, we get $d(u^G, v^G) \leq 2$ for any two distinct vertices $u^G$ and $v^G$ in $\Gamma_{sc}(G)$. Therefore, $\diam(\Gamma_s(G)) = 2$; a contradiction.

Thus, from  Case 1 and Case 2, we get $\diam(\Gamma_{sc}(G \times G)) \geq 3$. Hence, 
 $\Gamma_{sc}(G \times G) = 3$. 
\end{proof}

A dominating set of a graph $\Gamma$ is a subset $S$ of $V(\Gamma)$ such that every vertex  in $V(\Gamma) \setminus S$ is adjacent to at least one vertex in $S$.  The domination number of  $\Gamma$, denoted by $\lambda(\Gamma)$, is the  minimum cardinality of  dominating sets of  $\Gamma$.

\begin{prop}
Let $G$ be a  non-solvable group. Then  $\lambda(\Gamma_{sc}(G))=1$ if $|\Sol(G)|\neq 1$.
\end{prop}
\begin{proof}
Let $x$ be a nontrivial element in $\Sol(G)$. Then $x^G\in V(\Gamma_{sc}(G))$. Let $y^G \in V(\Gamma_{sc}(G)) \setminus \{x^G\}$ be an arbitrary vertex. Then $\langle x, y\rangle$ is solvable. Therefore,  $x^G$ and $y^G$ are adjacent. Hence, $\{x^G\}$ is a dominating set of $\Gamma_{sc}(G)$ and so $\lambda(\Gamma_{sc}(G))=1$.
\end{proof}

\section{Clique number}

The \emph{clique number} $\omega(\Gamma)$ of a graph $\Gamma$ is the number
of vertices in the largest complete subgraph of $\Gamma$. In this section
we investigate the clique number of the SCC-graphs of finite groups. The
main theorem of this section is that there are only finitely many finite groups
whose SCC-graph has a given clique number.

We begin with a theorem of Landau~\cite{Landau}.
Let $k(G)$ denote the number of conjugacy classes of the group $G$.

\begin{prop}
For any positive integer $m$, there are only finitely many finite groups
which have $k(G)=m$.
\label{p:cc}
\end{prop}

For let $x_1,\ldots,x_m$ be conjugacy class representatives, and let
$n_i=|C_G(x_i)|$ for $i=1,\ldots,m$. Then $|x_i^G|=|G|/n_i$; so
\[\sum_{i=1}^m\frac{1}{n_i}=1.\]
Now there are only finitely many expressions of $1$ as a sum of $m$ fractions
with unit numerator (this is ``folklore'', but is not a difficult exercise).
Moreover, the largest value of $n_i$ is $|C_G(1)|=|G|$.

Now we can deal with solvable groups.

\begin{thm}
There are only finitely many solvable groups $G$ for which $\Gamma_{sc}(G)$
has given clique number $d$.
\label{t:sol}
\end{thm}

\begin{proof}
By Theorem~\ref{t21}, if $G$ is solvable then $\Gamma_{sc}(G)$ is
complete, so its clique number is $k(G)-1$ (since the identity is
omitted from the graph); now Proposition~\ref{p:cc} finishes the result.
\end{proof}

We now give a result which will be used several times.

\begin{thm}\label{cl}
Let $G$ be a finite group. If $G$ has an element of order $n = \Pi_{i=1}^mp_i^{k_i}$, where $p_i$'s are distinct primes. Then $\Gamma_{sc}(G)$ has a clique of size \, $\Pi_{i=1}^m(k_i+1)-1$.
\end{thm}

\begin{proof}
Let $x\in G$ be an element of order $n$. Then $(x^r)^G \sim (x^s)^G$ for all proper divisors  $r, s$ of $n$. Since  total number of proper divisors of $n = \Pi_{i=1}^mp_i^{k_i}$ is $\Pi_{i=1}^m(k_i+1) - 1$, we get a clique in $\Gamma_{sc}(G)$ of size $\Pi_{i=1}^m(k_i+1)-1$.
\end{proof}

Next, we prove the main result for $d=2$, in a stronger form.

\begin{thm}
With the exception of the cyclic groups of orders $1$, $2$ and $3$ and the
symmetric group of degree~$3$, every finite group $G$ has the property that
$\Gamma_{sc}(G)$ contains a triangle (that is, has girth~$3$).
\end{thm}

\begin{proof}
If $G$ is solvable, then $k(G)=\omega(\Gamma_{sc}(G))+1$ (the extra $1$
coming from the identity of $G$), so $G$ has at most three conjugacy classes.
The groups listed in the theorem are all those having this property.

So we may assume that $G$ is non-solvable. If $G$ has an element whose order
is not a prime power, then some power (say $g$) of this element has order
$pq$, where $p$ and $q$ are distinct primes. Then $\Gamma_{sc}(G)$ contains a
clique of size~$3$, by Theorem~\ref{cl}.

So we may further assume that every element of $G$ has prime power order.

These groups were first studied by Higman~\cite{higman} in 1957;
Suzuki~\cite{suzuki} determined the simple groups with this property in 1965.
Subsequently all such groups have been classified~\cite{brandl,heineken}. The
story is somewhat tangled, perhaps due to the lack of a common name for the
class. Subsequently two names were proposed; a group with this property is
called a \emph{CP~group} by some authors, and an \emph{EPPO group} by others.
These groups have arisen in connection with other graphs defined on groups,
including the Gruenberg--Kegel graph (or prime graph) and the power graph:
see~\cite{cm}. The result we require is that a non-solvable group in which
every element has prime power order satisfies one of the following:
\begin{itemize}
\item[(a)] $G$ is one of $A_6$, $\PSL(2,7)$, $\PSL(2,17)$, $M_{10}$ or
$\PSL(3,4)$;
\item[(b)] $G$ has a nornal subgroup $N$ such that $G/N$ is $\PSL(2,4)$,
$\PSL(2,8)$, $\Sz(8)$ or $\Sz(32)$, and $N$ is a direct sum of copies of the
natural $G/N$-module over its field of definition.
\end{itemize}

Suppose first that we are in case (b). If we can find a
triangle in the solvable conjugacy class group of $G/N$, then it lifts to a
triangle in $\Gamma_{sc}(G)$. So it is enough to add the four possibilities
for $G/N$ to the list of groups in case (a).

In $\Sz(8)$, there are three conjugacy classes of elements of order $13$,
all represented in a cyclic subgroup of order $13$, giving us a triangle.
Similar arguments apply to $\Sz(32)$ (using an element of order $41$),
$\PSL(2,8)$ (order~$7$), and $\PSL(2,17)$ (order $3$ and two classes of
order~$9$). In $\PSL(2,4)$, a dihedral subgroup of order~$10$ meets two
conjugacy classes of elements of order $5$ and one class of involutions.
A similar argument applies to $A_6$ (using a dihedral group of order~$10$),
$\PSL(2,7)$ (using a non-abelian group of order $21$) $\PSL(3,4)$ (a
non-abelian group of order~$21$) and $M_{10}$ (a quaternion group of order~$8$
meets two conjugacy classes of elements of order~$4$ and one class of
involutions). All this information is easily obtained from the
\textit{$\mathbb{ATLAS}$ of Finite Groups}~\cite{atlas}.
\end{proof}

Now we come to the main result of the section.

\begin{thm}
For any positive integer $d$, there are only finitely many finite groups $G$
such that $\omega(\Gamma_{sc}(G))=d$.
\label{t:bound}
\end{thm}

This theorem can be regarded as a strengthening of Landau's result.

\begin{proof}
We assume that there are groups $G$ of arbitrarily large order such that
$\Gamma_{sc}(G)$ has clique number at most $d$, and aim for a contradiction.
We proceed in a number of steps.

\paragraph{Step 1} By Theorem~\ref{t:sol}, we can assume that $G$ is
non-solvable.

\paragraph{Step 2} We can assume that the solvable radical of $G$ is trivial.
For suppose that $\Sol(G)\ne1$ and that $|G/\Sol(G)|=m$, and suppose we know
that $m$ is bounded by a function of $d$. Then $\Sol(G)$ contains non-identity
elements from at most $d$ conjugacy classes of $G$, since these classes form a
clique in $\Gamma_{sc}(G)$. Since each such class splits into at most $m$
conjugacy classes in $\Sol(G)$, we see that $\Sol(G)$ has at most $dm$
non-trivial conjugacy classes, and hence has order bounded by a function of
$d$ and $m$. So if $m$ is also bounded by a function of $d$, then $|G|$ is
bounded by a function of $d$, as required.

\paragraph{Step 3} Let $G$ be a group with clique number bounded by $d$, and
let $S$ be the socle of $G$ (the product of the minimal normal subgroups). Then
$S$ is a product of non-abelian simple groups. We can assume that the number
of factors is bounded. For if we choose one non-identity element from each
factor, the chosen elements generate an abelian group; and elements of this
group which have different numbers of non-identity coordinates are not
conjugate in $G$.

\paragraph{Step 4} $C_G(S)=1$. For $C_G(S)$ is a normal subgroup of $G$, and
so contains a minimal normal subgroup, say $M$. But then
$M\le S\cap C_G(S)=Z(S)$, contradicting the fact that $S$ is a product
of centreless groups.

\paragraph{Step 5} It follows that $G$ acts faithfully on $S$ by conjugation.
Elements of $G$ permute the factors. Since their number is bounded, we can
assume that $G$ fixes all the factors; so the socle is simple, and $G$ is
almost simple.

\paragraph{Step 6} Now we invoke the Classification of Finite Simple Groups.
We can assume that $G$ is sufficiently large that its socle is not a sporadic
group. So there are three cases:
\begin{itemize}
\item[Case 1:] $S$ is alternating, so $G=A_n$ or $S_n$. Let $H$ be the subgroup
of $G$ with $\lfloor n/2\rfloor$ orbits of size $2$ (and one fixed point if
$n$ is odd). Then $H$ is abelian, and elements of $H$ may be products of any
even number of transpositions up to $2\lfloor n/4\rfloor$. But elements with
different numbers of transpositions are non-conjugate, so $\Gamma_{sc}(G)$
contains a clique of size $\lfloor n/4\rfloor$. Thus $n$ is bounded.
\item[Case 2:] $G$ is classical of large rank. 

Suppose first that the socle of $G$ is $\PSL(n,q)$ for large $n$. Then dropping
to a subgroup of index at most $2$, $G$ is the quotient of a subgroup of
$\PGamL(n,q)$ by the subgroup of scalar matrices. The diagonal matrices with
determinant~$1$ mod scalars form an abelian group of rank at least $n-2$, and
elements with different numbers of non-$1$ diagonal entries (greater than
$n/2$) are pairwise non-conjugate, giving a large clique.

Now suppose that $G$ is symplectic, unitary, or orthogonal. Then a cover of
$G$ acts on its natural module; this module is an orthogonal direct sum of $r$
hyperbolic planes and an isotropic space (see~\cite{taylor}), where $r$ is the
Witt index (and is at least $(n-2)/2$, where $n$ is the dimension of the 
module). So the cover of $G$ contains the direct product of $r$ copies of the
$2$-dimensional classical group, and the same contradiction is obtained.
\item[Case 3:] $G$ is of Lie type over a large field $\GF(q)$. In this case,
$G$ contains a subgroup $C$ of index at most $2$ in the multiplicative group of
the field (inside a subgroup of Lie rank~$1$, which is either $\PSL(2,q)$ or
$\Sz(q)$). Every conjugate of a generator $g$ of $C$ is the image of
$g$ or $g^{-1}$ under a field automorphism. So, if $q=p^r$ with $p$ prime,
then a conjugacy class of $G$ contains at most $2r$ generators. But altogether
there are at least $\phi((q-1)/2)$ generators, where $\phi$ is Euler's function.
The ratio of these two numbers tends to infinity with $q$. So a clique of size
larger than $d$ can be found if $q$ is sufficiently large.
\end{itemize}
\end{proof}

An alternative proof of the theorem runs as follows. Using arguments as above,
we reduce to the case where $G$ is a simple group. Then we apply a recent
result of Hung and Yang~\cite{hy}, asserting that the number of prime divisors
of a finite simple group is bounded above by a (quartic) function of the
maximum number of prime divisors of an element order. If the clique number of
the SCC-graph is bounded, then the number of prime divisors of an element order
is bounded, and hence the number of prime divisors of $|G|$ is bounded.

Next we claim that the prime divisors are bounded.
Let $p_1,\ldots,p_s$ be the prime divisors of $|G|$. We show by induction
on~$i$ that $p_i$ is bounded in terms of $d$. Since the solvable radical is
trivial, we may assume that $p_1=2$. Now let $i>1$ and take an element $x$
in~$G$ of order $p=p_i$; let $C=\langle x\rangle$. The $p-1$ generators of $C$
can lie in at most $d$ distinct conjugacy classes. Hence the cyclic group
$N_G(C)/C_G(C)$ has order at least $(p-1)/d$ (since an element of $G$
conjugating a generator of $C$ to another must normalize $C$).
On the other hand, every prime divisor of the order of $N_G(C)/C_G(C)$ divides
$p-1$ and therefore lies in $\{p_1,\ldots,p_{i-1}\}$. Moreover, the exponents of
the Sylow $p_j$-subgroups are bounded by $p_j^d$. So by induction,
$|N_G(C)/C_G(C)|$ is bounded in terms of $d$. Consequently, $p$ is bounded in
terms of $d$ as well.

Hence the exponent of $G$ is bounded. But there are only finitely many simple
groups with given exponent. (Jones~\cite{jones} showed that any infinite set
of the then known finite simple groups generates the variety of all groups, from
which it follows that their exponents are unbounded. Since then, only finitely
many further simple groups have been found, so the result is still valid. A
stronger result is proved by Babai, Goodman and Pyber~\cite{bgp}, namely,
there are only finitely many simple groups with a given largest prime divisor
of their order.) So we are done.

\medskip

We have not attempted to write down an explicit function bounding $|G|$ in
terms of the clique number of $\Gamma_{sc}(G)$.

\begin{cor}
Given $g$, there are only finitely many finite groups $G$ for which
$\Gamma_{sc}(G)$ can be embedded in a surface of genus~$g$.
\end{cor}

This holds because the genus of an embedding of the complete graph $K_n$ is
an unbounded function of $n$.

\section{Distance in SCC-Graph for locally finite group}
A locally finite group is a group for which every finitely generated subgroup is finite. An element of a group is said to be a $p$-element if the order of the element is a power of $p$, where $p$ is a prime. In this section we obtain some results on  distance between two vertices of $\Gamma_{sc}(G)$ for some locally finite groups,   analogous to certain results in \cite{hlm, ME2017}. 

\begin{prop}\label{cp1}
Let $G$ be a locally finite group. If $x,y\in G\setminus \{1\}$ are $p$-elements, where $p$ is a prime, then 
$d(x^G, y^G)\leq 1$.
\end{prop}
\begin{proof}
Since $G$ is a locally finite group and $x, y\in G\setminus \{1\}$ are $p$-elements, the subgroup $\langle x, y\rangle$ is finite. Let $P$ be a Sylow $p$-subgroup of  $\langle x, y\rangle$  containing $x$. Then $y^g = gyg^{-1} \in P$ for some $g \in G$ since all the Sylow $p$-subgroups are conjugate. Therefore,  $\langle x, y^g\rangle$ is solvable and so $d(x^G, y^G) \leq 1$. 
\end{proof}

\begin{prop}\label{coprime}
Let $G$ be a locally finite group. If $x, y \in G$ are of non-coprime orders, then 
$d(x^G,y^G)\leq 3$.
If either $x$ or $y$ is of prime order then $d(x^G,y^G)\leq 2$.
\end{prop}
\begin{proof}
Let $o(x) = pm$ and $o(y) = pn$, where $p$ is a prime and $m, n$ are positive integers. Then $x^m$ and $y^n$ are nontrivial $p$-elements of $G$. Therefore, by Proposition \ref{cp1}, we have
\[
d((x^m)^G, (y^n)^G)\leq 1.
\]
Clearly, $d(x^G,(x^m)^G)\leq 1$ and $ d((y^n)^G, y^G)\leq 1$. Therefore, if $x^G \ne y^G$ then $x^G \sim (x^m)^G \sim (y^n)^G \sim y^G$ is a path from $x^G$ to $y^G$. 
Hence, $d(x^G,y^G)\leq 3$.

Suppose that $o(x) = pm$ and $o(y) = p$. Then $x^m$ and $y$ are nontrivial $p$-elements of $G$. Therefore, by Proposition \ref{cp1}, we have
\[
d((x^m)^G, y^G)\leq 1.
\] 
Thus $x^G \sim (x^m)^G \sim y^G$ is a path from $x^G$ to $y^G$. 
Hence, $d(x^G,y^G)\leq 2$.
\end{proof}

\begin{prop}\label{cp}
Let $G$ be a locally finite group and $x, y \in G$. Suppose $p$ and $q$ are prime divisors of $o(x)$ and $o(y)$, respectively, and that $G$ has an element of order $pq$. Then
\begin{enumerate}
\item $d(x^G, y^G) \leq 5$, and moreover $d(x^G, y^G) \leq 4$ if either $x$ or $y$ is of prime power order.
\item If either a Sylow $p$-subgroup or a Sylow $q$-subgroup of $G$ is a cyclic or generalized quaternion finite group, then $d(x^G, y^G) \leq 4$. Moreover, $d(x^G, y^G) \leq 3$ if either $x$ or $y$ is of prime order.
\item If both Sylow $p$-subgroup and  Sylow $q$-subgroup of $G$ are either cyclic or generalized quaternion finite groups, then $d(x^G, y^G) \leq 3$. Moreover, $d(x^G, y^G) \leq 2$ if either $x$ or $y$ is of prime order.
\end{enumerate}
\end{prop}
\begin{proof}
Let $o(x) = pm$ and $o(y) = qn$ for some positive integers $m, n$. 
Let  $a \in G$ be an element of order $pq$. Then $o(a^q) = p$ and $o(a^p) = q$. Also, $a^p$ commutes with $a^q$.
\begin{enumerate}
\item We have 
\begin{equation*}
d(x^G,(x^m)^G)\leq 1, \,\,  d((a^q)^G,(a^p)^G)=1, \text{ and } d((y^n)^G, y^G)\leq 1.
\end{equation*}
Since $o(x^m) = o(a^q) = o(y^n) = p$, by Proposition \ref{cp1}, we have
\begin{equation*}
d((x^m)^G,(a^q)^G)\leq 1 \text{ and } d((a^p)^G,(y^n)^G)\leq 1.
\end{equation*}
Therefore, $d(x^G, y^G)\leq 5$.

If  $o(x) = p^s$  for some positive integer $s$ then, by Proposition \ref{cp1}, we have  $d(x^G, (a^q)^G) \leq 1$. Similarly, if   $o(y) = q^t$  for some positive integer $t$ then  $d(y^G, (a^p)^G) \leq 1$. Therefore, $d(x^G, y^G)\leq 4$.
\item\label{cyq} Without any loss of generality assume that Sylow $p$-subgroup of $G$ is either a cyclic group or a generalized quaternion finite group.  Let $P$ and $Q$ be two Sylow $p$-subgroups of $G$ containing $x^m$ and $a^q$ respectively. Since $P$ is finite, by \cite[Theorem 14.3.4]{Robinson}, $Q$ is also finite and $P = gQg^{-1}$ for some $g \in G$ and so $ga^qg^{-1} \in P$. Therefore, $\langle x^m \rangle$ and $\langle ga^qg^{-1} \rangle$ are subgroups of $P$ having order $p$. Since $P$ is cyclic  or a generalized quaternion  group, by \cite[Theorem 5.3.6]{Robinson}, we have $\langle x^m \rangle =\langle ga^qg^{-1} \rangle$. Therefore, $ga^qg^{-1} = (x^m)^i$ for some integer $i$ and so  $\langle x, ga^qg^{-1}\rangle = \langle x, (x^m)^i\rangle = \langle x \rangle$.  Hence $d(x^G, (a^q)^G) \leq 1$. We also have
\begin{equation*}
d((a^q)^G,(a^p)^G)=1, \,\, d((a^p)^G,(y^n)^G)\leq 1,\text{ and } d((y^n)^G, y^G)\leq 1.
\end{equation*}
Thus $d(x^G, y^G)\leq 4$.

If $o(x) = p$ then $\langle x\rangle =\langle ga^qg^{-1}\rangle$. Therefore, $x= ga^{qt}g^{-1}$ for some integer $t$. We have $x^G = (a^{qt})^G$ and so $\langle a^{qt}, a^p\rangle$ is abelian. Hence, $d(x^G, (a^p)^G) \leq 1$ and so $d(x^G, y^G)\leq 3$.   
\item If both  Sylow $p$-subgroup and  Sylow $q$-subgroup of $G$ are either cyclic or generalized quaternion finite groups, then proceeding as part \ref{cyq} we get
\begin{equation*}
d(x^G,(a^q)^G)\leq 1,\,\, d((a^q)^G,(a^p)^G)=1, \text{ and } d((a^p)^G, y^G)\leq 1.
\end{equation*}
Therefore, $d(x^G, y^G)\leq 3$.

If $o(x) = p$ then proceeding as in part (b), we have $d(x^G, (a^p)^G) \leq 1$ and so $d(x^G, y^G)\leq 2$.
\end{enumerate}
\end{proof}

We conclude this section with the following consequence.

\begin{thm}
Let $G$ be a finite group. Let  $H$ and $K$ be two subgroups of $G$ such that $H$ is  normal in $G$,  $G = HK$ and $\Gamma_{sc}(H), \Gamma_{sc}(K)$ are connected. If there exist two elements $h \in H \setminus \{1\}$ and $x \in G \setminus H$ such that $h^G$ and $x^G$ are connected in $\Gamma_{sc}(G)$, then $\Gamma_{sc}(G)$ is connected.
\end{thm}
\begin{proof}
Let $a, b\in G$ such that $a^G$ and $b^G$ are two distinct vertices in $\Gamma_{sc}(G)$. 

If $a, b\in H$ then there exists a path from $a^H$ to $b^H$, since $\Gamma_{sc}(H)$ is connected. Hence, $a^G$ and $b^G$ are connected.  Let  $a \notin H$ and $o(a) = n$. Let $f: G/H \to K/(H\cap K)$ be an isomorphism  and $f(aH) = x(H\cap K)$, where $x \in K$. Then $x^n(H\cap K) = f(a^nH) = H\cap K$ and so $x^n \in H\cap K$. Let $d = \gcd(o(a), |K|)$. Then there exist integers $r, s$ such that 
\[
x^d = x^{nr + |K|s} = (x^n)^r.(x^{|K|})^s \in H\cap K.
\] 
Therefore, $d > 1$. Let $p$ be a prime divisor of $d$. Then there exists an element $k_1 \in K$ such that $\gcd(o(a), o(k_1)) \neq 1$.  Therefore, by Proposition \ref{coprime}, there is a path from $a^G$ to $k_1^G$. Similarly, if $b \notin H$ then there exists an element $k_2 \in K$ such that there is a path from $k_2^G$ to $b^G$. We have $k_1^G = k_2^G$ or there is a path from $k_1^K$ to $k_2^K$, since $\Gamma_{sc}(K)$ is connected. Therefore, $k_1^G = k_2^G$ or there is a path from $k_1^G$ to $k_2^G$. Thus  $a^G$ and $b^G$ are connected. If $b \in H$ then, by  given conditions, there exist two elements $h \in H \setminus \{1\}$ and $x \in G \setminus H$ such that  there is a path from $x^G$ to $h^G$ and a path from $h^G$ to $b^G$ (since $\Gamma_{sc}(H)$ is connected). Since $x \notin H$, proceeding as above we get a path from $x^G$ to $k_3^G$ for some $k_3 \in K$ and hence a path from $a^G$ to $x^G$. Thus we get a path from $a^G$ to $b^G$. Hence, $\Gamma_{sc}(G)$ is connected. 
\end{proof}

\subsection*{Acknowledgements}

The fourth author is supported by the German Research Foundation (SA 2864/1-2 and SA 2864/3-1).

The authors declare no conflict of interest.

No data were generated or used in the course of this project.


\begin{thebibliography}{33}


\bibitem{acns}
G. Arunkumar, P. J. Cameron, R. K. Nath and L. Selvaganesh, Super graphs on groups, I, available at https://arxiv.org/abs/2112.02395.

\bibitem{bgp}
L. Babai, A. J. Goodman and L. Pyber,
Groups without faithful transitive permutation representations of small degree,
\textit{J. Algebra} \textbf{195} (1997), 1--29.

\bibitem{BHM90} 
 E. A. Bertram, M. Herzog and A. Mann,   On a graph related to conjugacy classes of groups, {\it Bull. London Math. Soc.} \textbf{22}, 569--575, 1990.
 

\bibitem{bn-nath-2020} 
 P. Bhowal, D. Nongsiang and R. K.  Nath, Solvable graphs of finite groups,  {\it Hacet. J. Math. Stat.},  \textbf{49}(6), 1955--1964, 2020.

\bibitem{brandl}
R. Brandl,
Finite groups all of whose elements are of prime power order,
\textit{Bolletin Unione Matematica Italiana} (5)
\textbf{18} A (1981), 491--493

\bibitem{BF1955} 
R. Brauer and K. A. Fowler,  On groups of even order,  {\it Ann.  Math.} \textbf{62}(2), 565--583, 1955.

\bibitem{burnside}
W. Burnside, \textit{Theory of Groups of Finite Order} (second edition), 
Cambridge Univ. Press, Cambridge 1911, reprinted Dover Publ., New York, 1955.

\bibitem{gong}
Peter J. Cameron, Graphs defined on groups,
\textit{International J. Group Theory} \textbf{11} (2022), 43--124.

\bibitem{cm}
Peter J. Cameron and Natalia Maslova,
Criterion of unrecognizability of a finite group by its Gruenberg--Kegel graph,
\textit{J. Algebra}, in press.

\bibitem{atlas}
J. H. Conway, R. T. Curtis, S. P. Norton, R. A. Parker and R. A. Wilson,
\textit{$\mathbb{ATLAS}$ of Finite Groups}, Clarendon Press, Oxford, 1985.


\bibitem{dghp} 
S. Dolfi, R. M. Guralnick, M. Herzog and C. E. Praeger,  A new solvability criterion for finite groups, {\em J. London Math. Soc.} \textbf{85}(2), 269--281, 2012.


\bibitem{gkps} 
R. Guralnick, B. Kunyavskii, E. Plotkin and A. Shalev,  Thompson-like characterizations of the solvable radical, {\em J. Algebra} \textbf{300}, 363--375, 2006.




\bibitem{heineken}
Hermann Heineken,
On groups all of whose elements have prime power order,
\textit{Math. Proc. Royal Irish Acad.} \textbf{106} (2006), 191--198.

\bibitem{hlm} 
M. Herzog, P. Longobardi and M. Maj,  On a commuting graph on conjugacy classes of groups, {\it Comm. Algebra} \textbf{37}, 3369--3387, 2009.

\bibitem{higman}
Graham Higman,
Finite groups in which every element has prime power order,
\textit{J. London Math. Soc.} (1) \textbf{32} (1957), 335--342.

\bibitem{hy}
Nguyen Ngoc Hung and Yong Yang,
On the prime divisors of element orders,
\textit{Math. Nachrichten}, in press.

\bibitem{jones}
G. A. Jones,
Varieties and simple groups,
\textit{J. Australian Math. Soc.} 17 (1974), 163--173.

\bibitem{Landau}
E. Landau, \"{U}ber die {K}lassenzahl der bin\"aren quadratischen
  {F}ormen von negativer {D}iscriminante, \textit{Math. Ann.} \textbf{56} (1903),
  671--676.

\bibitem{ME2017} 
A. Mohammadian and A. Erfanian,  On the nilpotent conjugacy class graph of groups, {\it Note  Mat.} {\bf 37}(2), 77--89, 2017.





\bibitem{EN76} 
B. H. Neumann,  A problem of Paul Erd$\ddot{\rm o}$s on groups, \emph{J. Aust. Math. Soc.} {\bf 21}, 467--472, 1976.

\bibitem{Robinson} 
D. J. S. Robinson, A Course in the Theory of Groups, Second Edition, Berlin, Springer-Verlag, 1982.

\bibitem{suzuki}
M. Suzuki,
Finite groups in which the centralizer of any element of order $2$ is
$2$-closed,
\textit{Ann. Math.} (2) \textbf{82} (1965), 191--212.

\bibitem{taylor}
D. E. Taylor, \textit{The Classical Groups}, Heldermann Verlag, Berlin, 1992.

\end{thebibliography}
\end{document}